\documentclass[12pt]{amsart}
\usepackage{latexsym,amsmath,amssymb}
\usepackage{graphicx}

\title[Game Tree Approach]{\protect{A Game-Tree approach to discrete infinity Laplacian with Running Costs}}

\author{Qing Liu}\author{Armin Schikorra}

\address{Qing Liu, Department of Mathematics, University of Pittsburgh, Pittsburgh, PA 15260, USA, {\tt qingliu@pitt.edu}}
\address{Armin Schikorra, Max-Planck Institut MiS Leipzig, Inselstr. 22, 04103 Leipzig, Germany, {\tt armin.schikorra@mis.mpg.de}}

\thanks{A.S. is supported by DAAD fellowship D/12/40670.}

plus 6pt minus 12pt
\def\ilap{\Delta_\infty}
\def\eps{\varepsilon}

\def\R{{\mathbb R}}

\def\N{{\mathbb N}}

\def\T{{\mathbb T}}

\newtheorem{theorem}{Theorem}
\newtheorem{mythm}{Theorem}

\newtheorem{lemma}[theorem]{Lemma}
\newtheorem{corollary}[theorem]{Corollary}
\newtheorem{proposition}[theorem]{Proposition}

  \theoremstyle{definition}
\newtheorem{remark}[theorem]{Remark}
\newtheorem{definition}[theorem]{Definition}
\newtheorem{example}[theorem]{Example}

\def\diam{{\rm diam\,}}

\newcommand{\brac}[1]{\left (#1 \right )}

\newcommand{\barint}{
\rule[.036in]{.12in}{.009in}\kern-.16in \displaystyle\int }

\newcommand{\barcal}{\mbox{$ \rule[.036in]{.11in}{.007in}\kern-.128in\int $}}

\def\mvint_#1{\mathchoice
          {\mathop{\vrule width 6pt height 3 pt depth -2.5pt
                  \kern -8pt \intop}\nolimits_{\kern -3pt #1}}%
          {\mathop{\vrule width 5pt height 3 pt depth -2.6pt
                  \kern -6pt \intop}\nolimits_{#1}}%
          {\mathop{\vrule width 5pt height 3 pt depth -2.6pt
                  \kern -6pt \intop}\nolimits_{#1}}%
          {\mathop{\vrule width 5pt height 3 pt depth -2.6pt
                  \kern -6pt \intop}\nolimits_{#1}}}

\numberwithin{theorem}{section} \numberwithin{equation}{section}

\DeclareMathOperator*{\stinfsup}{\operatorname{EXT}}

\DeclareMathOperator*{\infsup}{infsup}

\begin{document}

\sloppy

\sloppy


\begin{abstract}
We give a self-contained and elementary proof for boundedness, existence, and uniqueness of solutions to dynamic programming principles (DPP) for biased tug-of-war games with running costs. The domain we work in is very general, and as a special case contains metric spaces. Technically, we introduce game-trees and show that a discretized flow converges uniformly, from which we obtain not only the existence, but also the uniqueness. Our arguments are entirely deterministic, and also do not rely on (semi-)continuity in any way; in particular, we do not need to mollify the DPP at the boundary for well-posedness.
\end{abstract}
\maketitle

\section{Introduction}
Let $(X,d)$ be a metric space of finite diameter, and let $Y \subsetneq X$ be any nonempty, proper subset. With $B_\eps(x)$ we denote the balls centered at $x$ with $d$-radius $\eps$. For simplicity, let us assume for the introduction that these are the open balls; Later we see that all the results presented here also hold for closed balls, and we even can treat much more general sets $B(x)$, cf. Definition~\ref{def:admissible}.

Given running costs $f = \frac{1}{2}\eps^2 \tilde{f}: Y \to \R$ and boundary values $F: X \backslash Y \to \R$, for $\mu \in (0,1)$, and $\eps > 0$, we are interested in the analysis of solutions $u: X \to \R$ to the following Dynamic Programming Principle (DPP)
\begin{equation}\label{eq:DPP}
 \begin{cases}
  u(x) = \mu \sup\limits_{B_\eps(x)} u + (1-\mu) \inf\limits_{B_\eps(x)} u + \frac{1}{2}\eps^2 \tilde{f}(x) \quad &\mbox{if $x \in Y$},\\
 u(x) = F(x) \quad &\mbox{if $x \in X \backslash Y$}.
  \end{cases}
\end{equation}
In PDE-terms, the set $Y$ plays the role of a domain, and $X \backslash Y$ plays the role of the boundary of $Y$.

If one, e.g., thinks of $Y$ as a domain in some euclidean space $X = \R^n$, then as shown for $f \equiv 0$ in \cite{PPS10} with $\mu = 1/2-\beta\eps/4$, this can be seen as a discretization of the PDE
\begin{equation}\label{eq:ilap}
 \ilap u + \beta |\nabla u| = \tilde{f}(x),
\end{equation}
which was our main motivation for considering this particular DPP, see also \cite{QAPerron}.

We show that if $\inf_Y f > 0$, $\sup_X |F| + \sup_X|f| < \infty$, then there exists a unique solution $u : X \to \R$ to \eqref{eq:DPP}. 

In fact, we prove that the solution $u$ to \eqref{eq:DPP} is the uniform limit of the sequence $u_k: X \to \R$, which is obtained by the following iteration starting from an arbitrary $u_0 : X \to \R$ with $\sup_X |u_0| < \infty$:
\begin{equation}\label{eq:iteration}
 \begin{cases}
  u_{k+1}(x) = \mu \sup\limits_{B_\eps(x)} u_k + (1-\mu) \inf\limits_{B_\eps(x)} u_k + f(x) \quad &\mbox{if $x \in Y$},\\
 u_{k+1}(x) = F(x) \quad &\mbox{if $x \in X \backslash Y$}.
  \end{cases}
\end{equation}
In some sense, \eqref{eq:iteration} can be interpreted as a discrete version of the following flow for $u: [0,\infty) \times X \to \R$
\[
 \begin{cases}
  u_t = \ilap u  + \beta |\nabla u| - \tilde{f}(x) \quad &\mbox{in $Y \times (0,\infty)$}\\
  u = F \quad &\mbox{in $X \backslash Y \times (0,\infty)$}\\
  u(0,\cdot) = u_0(\cdot) \quad &\mbox{in $Y$}.
 \end{cases}
\]
Our results therefore imply that the discretized flow starting from any $u_0: X \to \R$ converges to a solution of the discrete version of \eqref{eq:ilap}.

A flow-approach was also applied to a stationary Neumann boundary problem in \cite{APSS}. The authors considered a long-time limit of the value function associated with a time-dependent tug-of-war game on graphs and smooth domains. We however treat a distinct problem with the iteration method, very different from their probability approach.

The iteration \eqref{eq:iteration} is inspired by the recent article \cite{LuiroParviainenSaksman}, where the authors considered the following DPP for $\alpha \in (0,1]$
\[
  u(x) = (1-\alpha) \brac{\frac{1}{2} \sup\limits_{B_\eps(x)} u + \frac{1}{2} \inf\limits_{B_\eps(x)} u} + \alpha \mvint_{B_\eps(x)} u.
\]
They showed uniform convergence for the iteration starting from Borel measurable functions $u_0$. 
Nevertheless, their arguments rely crucially on the assumption $\alpha > 0$. Since we deal with the case of $\alpha = 0$ and positive running costs $f$, our techniques are different.


We also obtain results for the DPP-version of super- and subsolutions,
\begin{definition}[Super and Sub-Solutions to \eqref{eq:DPP}]\label{def:subsolution}
We say that $u : X \to \R$ is a supersolution if
\[ 
\begin{cases}
  u(x) \geq \mu \sup\limits_{B_\eps(x)} u + (1-\mu) \inf\limits_{B_\eps(x)} u + f(x) \quad &\mbox{if $x \in Y$},\\
 u(x) = F(x) \quad &\mbox{if $x \in X \backslash Y$},
\end{cases}
\]
and a subsolution if
\[ 
\begin{cases}
  u(x) \leq \mu \sup\limits_{B_\eps(x)} u + (1-\mu) \inf\limits_{B_\eps(x)} u + f(x) \quad &\mbox{if $x \in Y$},\\
 u(x) = F(x) \quad &\mbox{if $x \in X \backslash Y$}.
\end{cases}
\]
As usual, a function $u$ is a solution if and only if it is both, a subsolution and a supersolution.
\end{definition}
Note that $u \equiv \pm \infty$ in $Y$ is a subsolution and supersolution. All our Theorems will exclude this case.

Also, one observes that if $u_0$ in \eqref{eq:iteration} is a subsolution, then pointwise $u_{k+1} \geq u_k$ for all $k \in \N_0$, and if $u_0$ is a supersolution, then $u_{k+1} \leq u_k$ for all $k \in \N_0$.

Our first result is the uniform boundedness of solutions to \eqref{eq:iteration}, as well for subsolutions as also for supersolutions:
\begin{mythm}[Boundedness]\label{th:boundedness}
For any $\Lambda > 0$, $\mu \in (0,1)$, there exists $C = C(\mu,\Lambda) > 0$ such that the following holds: for any
$u_k : X \to \R$, $k \in \N_0$, such that
\begin{equation}\label{eq:supXu0linfty}
 \sup_X u_0 < \infty,
\end{equation}
and
\[
 \begin{cases}
  u_{k+1}(x) \leq \mu \sup\limits_{B_\eps(x)} u_k + (1-\mu) \inf\limits_{B_\eps(x)} u_k + \Lambda \quad &\mbox{if $x \in Y$},\\
 u_{k+1}(x) \leq \Lambda \quad &\mbox{if $x \in X \backslash Y$},
  \end{cases}
\]
we have
\[
 \limsup_{k \to \infty} \sup_X u_k \leq C.
\]
In particular, any subsolution $\underline{u}: X \to \R$ with $\sup_X \underline{u} < \infty$ satisfies
\[
 \sup_X \underline{u} \leq C,
\]
and any supersolution $\bar{u}: X \to \R$ with $\inf_X \bar{u} > -\infty$ satisfies
\[
 \inf_X \bar{u} \geq -C.
\]
\end{mythm}
Theorem~\ref{th:boundedness} is a special case of Theorem~\ref{th:boundednessGen} in Section~\ref{s:bounded}.

In \cite{QAPerron} we show a similar boundedness result with different methods for more general DPP's, but only for sub- and supersolutions.

\begin{mythm}[Uniform Convergence]\label{th:uniformconvergence}
Fix $\mu \in (0,1)$, $f,F: X \to \R$, such that
\[
 \sup_X |F| + \sup_X |f| < \infty,
\]
and
\[
 \inf_Y f > 0.
\]
Then there exists $u: X \to \R$, such that $u_k$ converges uniformly to $u$, for any sequence $u_k: X \to \R$ as in \eqref{eq:iteration}
with
\[
 \sup_X |u_0| < \infty.
\]
\end{mythm}

Technically, in order to prove Theorem~\ref{th:uniformconvergence}, we introduce the concept of Game Trees, which encode the optimal game progression of two players which want to maximize and minimize the value function $u$, respectively. To our best knowledge this is a new approach.

Then, an estimate reminiscent of a comparison principle for game-trees, Proposition~\ref{pr:depthest}, and an argument reminiscent of semi-group properties, Lemma~\ref{la:convergencegen}, are used. All the arguments are completely elementary and eventually rely on iteration estimates for sequences and series.
\begin{figure}
\centering
\includegraphics[width=130mm]{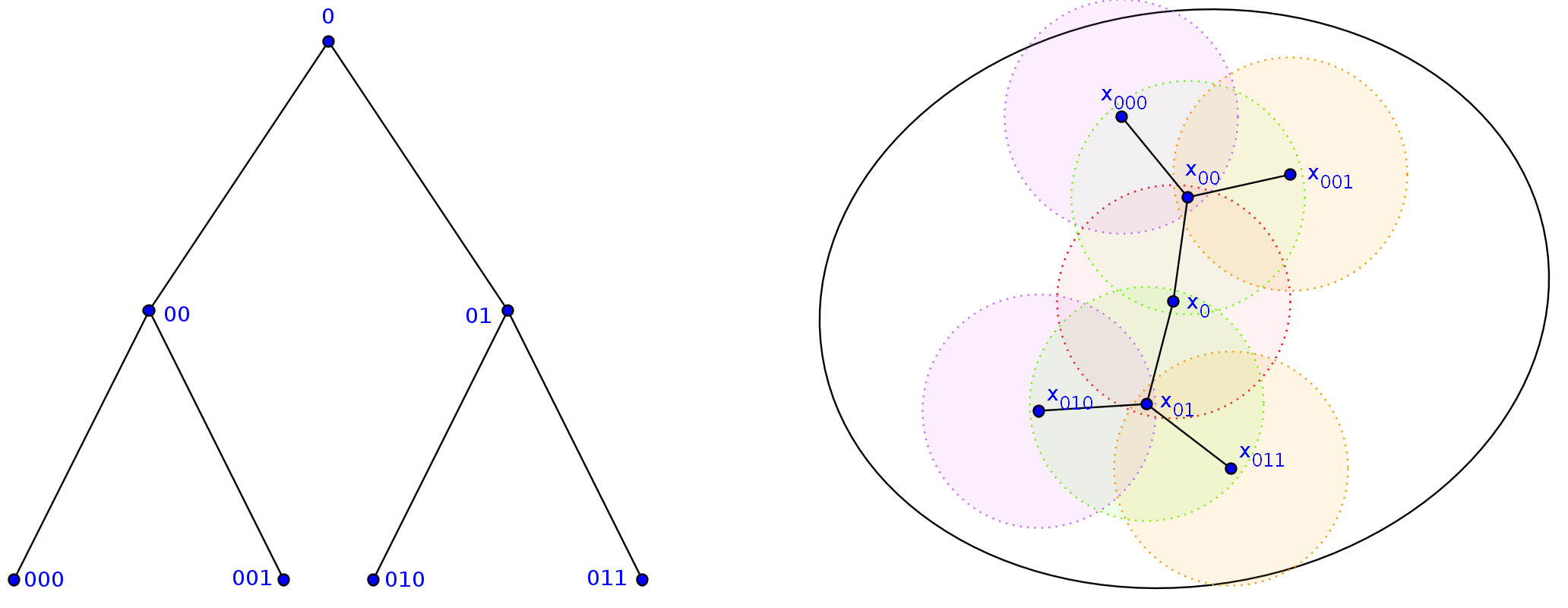}
\caption{A binary tree $T$ and a sequence $(x_t)_{t \in T}$. The circles denote $B(x_t)$. Philosophically, the sequence $(x_t)$ encodes the strategies of two players.}
\label{fig:trees}
\end{figure}

Let us remark on other known approaches to obtain existence and/or uniqueness of solutions $u : X \to \R$ to DPPs related to tug-of-war games: 
There is a stochastic game argument, relying on Kolmogorov's Theorem of probability measures on infinite dimensional spaces, cf. \cite{PPS10,PSSW09}. See also the stochastic game interpretation for the $p$-Laplacian in \cite{PS,MPR1}.
Another approach for existence, in \cite{QAPerron} we extend Perron's method to the discrete setting. Note, however, that our argument here is constructive using a flow.
In \cite{GA98,G07}, using also an iteration technique, they obtained continuous solutions to a modified DPP with shrinking balls near the boundary. In \cite{ArmstrongSmart1,ArmstrongSmart2}, existence and uniqueness for a related situation is obtained. Their DPP is mollified towards the boundary, and some of their arguments rely on semicontinuity of sub- or supersolutions.

For our situation, where we do not mollify the DPP towards the boundary, one cannot expect even semicontinuity, as the two following examples show. 
\begin{example}\label{ex:noncontinuity}Let $X = \R$, $Y = (0,2)$, $\mu = \frac{1}{2}$, $\eps = 1$. Define $F(x) = 0$ on $(-\infty,0]$, and $F(x) = 1$ on $[2,\infty)$. 

For $f \equiv 0$ and $f \equiv 1$, there are respective functions $u_0,u_1: X \to \R$ which are not semi-continuous and nevertheless satisfy
\[
 u(x) = \begin{cases}
         \frac{1}{2} \sup\limits_{(x-1,x+1)} u + \frac{1}{2} \inf\limits_{(x-1,x+1)} u + f\quad &\mbox{in $(0,2)$},\\
         u = 0 \quad &\mbox{in $(-\infty,0]$},\\
         u = 12 \quad &\mbox{in $[2,\infty)$}.
        \end{cases}
\]

\begin{itemize}
\item For $f \equiv 0$, we take
\[
 u_0(x) = \begin{cases}
         4 \quad &x \in (0,1),\\
         6 \quad &x = 1,\\
         8 \quad &x \in (1,2),\\
        \end{cases}
\]
\item if $f \equiv 1$, we take
\[
 u_1(x) = \begin{cases}
         6 \quad &x \in (0,1),\\
         9 \quad &x = 1,\\
         10 \quad &x \in (1,2).\\
        \end{cases}
\]
\end{itemize}
\end{example}

Generally, obtaining uniqueness is more difficult than obtaining existence, but here it follows from uniform convergence.

\begin{corollary}[Uniqueness]
For any $F: X \to \R$, $f: X \to \R$ satisfying
\[
 \sup_X |F| + \sup_X |f|  < \infty,
\]
and
\[
 \inf_Y f > 0,
\]
there is exactly one solution to \eqref{eq:DPP}.
\end{corollary}
\begin{proof}
Let $u$ be the solution from Theorem~\ref{th:uniformconvergence}, and let $\tilde{u}$ be any other solution. Starting the iteration \eqref{eq:iteration} with $u_0 := \tilde{u}$, and thus $u_k = \tilde{u}$, we obtain from Theorem~\ref{th:uniformconvergence}
\[
 0 = \lim_{k \to \infty} \sup_X |u - u_k| = \sup_X |u - \tilde{u}|.
\]
\end{proof}

In \cite{QAPerron} we obtain a comparison principle for more general DPPs, for \emph{strict} super -and subsolutions. Generally we remarked there, that uniqueness is equivalent to a comparison principle for all super- and subsolutions, so it is not surprising that we have

\begin{corollary}[Comparison Principle]
Given $F: X \to \R$, $f: X \to \R$ satisfying
\[
 \sup_X |F| + \sup_X |f|  < \infty,
\]
and
\[
 \inf_Y f > 0,
\]
let $\underline{v}$ be a subsolution, and $\bar{v}$ be a supersolution to \eqref{eq:DPP} in the sense of Definition~\ref{def:subsolution} and $\sup_X |\underline{v}| + \sup_X |\bar{v}| < \infty$. Then, pointwise $\underline{v} \leq \bar{v}$.
\end{corollary}
\begin{proof}
Start the iteration \eqref{eq:iteration} $\underline{v}_k$, $\bar{v}_k$ from $\underline{v}$ and $\bar{v}$, respectively. Then $\underline{v}_k \leq \underline{v}_{k+1}$ and $\bar{v}_k \geq \bar{v}_{k+1}$. In particular,
\[
 \bar{v}(x) \geq \limsup_{k \to \infty} \bar{v}_k(x),
\]
and
\[
 \underline{v}(x) \leq \liminf_{k \to \infty} \underline{v}_k(x).
\]
On the other hand, by Theorem~\ref{th:uniformconvergence}, there exists a solution $u: X \to \R$ and
\[
\underline{v}(x) \leq \lim_{k \to \infty} \underline{v}_k (x) =  u(x) = \lim_{k \to \infty} \bar{v}_k (x)  \leq \bar{v}(x).
\]
\end{proof}
For the tug-of-war DPP without running costs \cite{AS} showed a comparison principle under continuity assumptions. In our case, we do not assume any regularity at all, and as Example~\ref{ex:noncontinuity} shows, one cannot hope for even lower- or upper semicontinuity for sub- or supersolutions.

Our arguments and theorems hold true on more general spaces than described above, indeed we are going to show them for the following setting (where $B(x)$ replaces the role of $B_\eps(x)$).
\begin{definition}[Admissible Setups]\label{def:admissible}

Let $X$ be a set, and $Y \subset X$. Moreover associate to any $x \in Y$ a set $B(x) \subset X$. We say that the collection 
\[ \left (X, Y, \{B(x),x \in Y\} \right )
\] is admissible, if the following holds

\begin{itemize}
 \item $X$, $X \backslash Y$, $Y$ and $B(x)$ are nonempty for all $x \in Y$,
 \item There exists a finite integer, which we shall call the diameter of $X$ and denote by $\diam X \in \N$, such that for any $x \in Y$ there exists $d = d(x) < \diam X$ and a chain of $(x_i)_{i = 0}^{d} \in X$, such that $x_0 = x$ and $x_{d} \in X \backslash Y$, and $x_i \in B(x_{i-1})$ for all $i \in \{1,\ldots,d\}$.
\end{itemize}

\begin{figure}
\centering
\includegraphics[width=130mm]{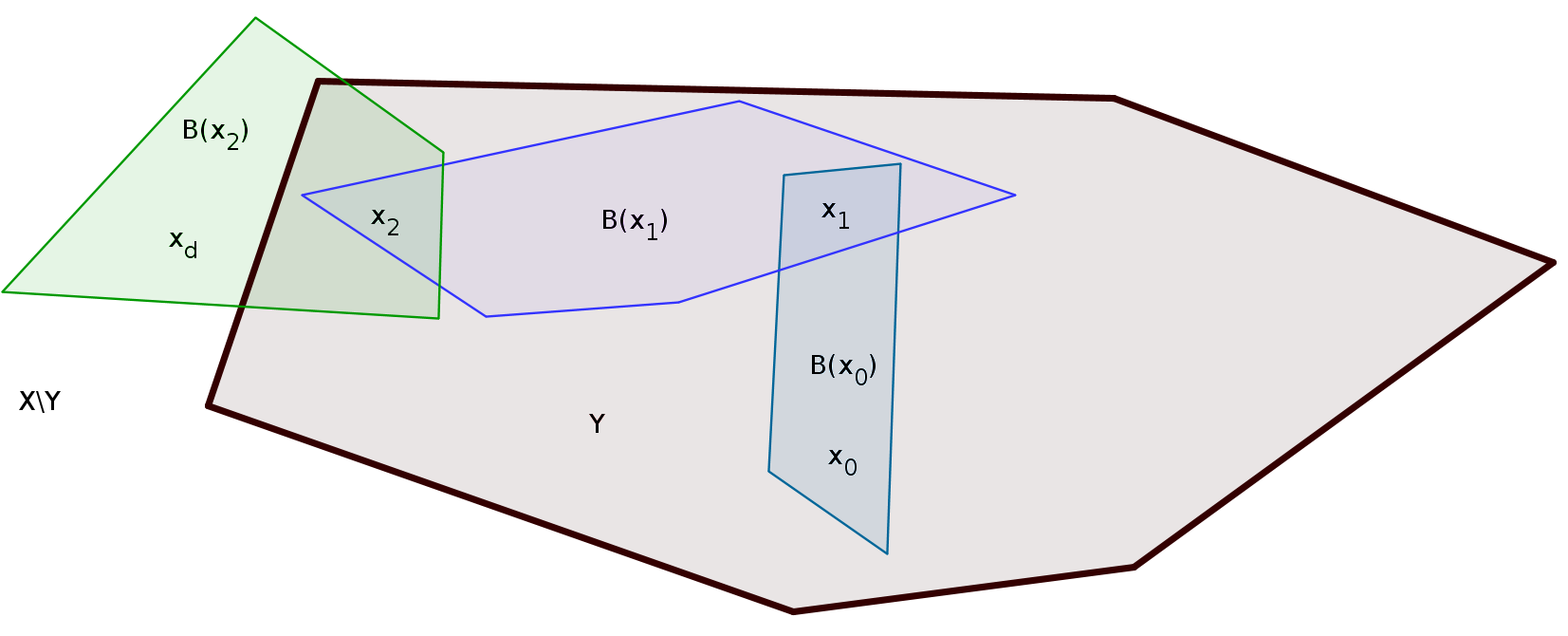}
\caption{One possible choice of $B(x)$ connecting $x_0$ to $x_d$}
\end{figure}

\end{definition}
\begin{remark}
Note that in particular we do not need symmetry-conditions such as $x \in B(y) \Rightarrow y \in B(x)$.

Also, it is a straight-forward generalization of our arguments to use two different families of ``balls'' $B(x)$, one for the $\sup$-term and another for the $\inf$-term, which could be completely different. For the sake of simplicity of notation, we leave this as an exercise.

\subsection*{Acknowledgment} We would like to thank Juan Manfredi, Marta Lewicka, Yoshikazu Giga for their interest in our results and helpful comments on the draft. Also we would like to thank Scott Armstrong to point out the relations to \cite{ArmstrongSmart1,ArmstrongSmart2}. The graphics have been done using GeoGebra\footnote{www.geogebra.org}.
\end{remark}

\section{Iteration Estimates and Trees}
In this section we first prove some estimates on sequences and series, and then introduce some terminology on the trees we are using. The arguments in this section are quite elementary, albeit not obvious. 

The main theme of this section could be described as adapting the following well-known iteration argument, cf., e.g., \cite[Chapter III, Lemma 2.1.]{GiaquintaMI}, to our needs.
\begin{proposition}\label{pr:standarditeration}
Let $(a^k)_{k \in \N}$ be a non-negative sequence with the rule that for some $\theta \in (0,1)$, $\Lambda > 0$ we have
\[
 a^{k+1} \leq \theta\ a^k + \Lambda.
\]
Then, for a constant $C$ depending only on $\theta$ and $\Lambda$, 
\[
 a^{k} \leq C a^0 + C,
\]
and
\[
 \limsup_{k \to \infty} a^{k} \leq C.
\]
\end{proposition}

\subsection{Iteration on Systems}
Our first proposition can be described as a version of Proposition~\ref{pr:standarditeration} for systems of sequences.
\begin{proposition}\label{pr:akalphaest}
For any $d \in \N$, any $\Lambda > 0$ and any $\mu \in (0,1)$ there exists a constant $C = C(d,\Lambda,\mu) > 0$ such that for any sequence $(a^k_\alpha)_{\alpha = 0,\ldots,d; k \in \N} \subset \R_+$ satisfying
\begin{equation}\label{eq:akp1iteration}
 a^{k+1}_\alpha \leq \begin{cases}
                 \Lambda \quad & \mbox{if }\alpha = 0,\\
                 \mu a_d^k + (1-\mu) a_{\alpha - 1}^k + \Lambda \quad & \mbox{if }\alpha = 1,\ldots,d.
                 \end{cases}
\end{equation}
we have
\[
 \max_{\alpha} \limsup_{k \to \infty} a^k_\alpha \leq C.
\]
\end{proposition}
\begin{proof}
We are going to show for $d \geq 2$ that for a certain choice of $\lambda_\alpha > 0$, setting
\[
  b^k := \sum_{\alpha=1}^d \lambda_\alpha a^k_\alpha,
\]
there is $\theta \in (0,1)$ such that
\begin{equation}\label{eq:bkp1wish}
 b^{k+1} \leq \theta b^k + C_0.
\end{equation}
This implies the claim, since by Proposition~\ref{pr:standarditeration} from the above we obtain for a constant $C_1$ depending on $\theta$ and $C_0$,
\[
 \sup_{k} b^k \leq C_1 (1 + b_0) < \infty.
\]
In particular, $\limsup_{k \to \infty} b^k < \infty$, and thus
\[
 \limsup_{k \to \infty} b^{k} \leq \theta \limsup_{k \to \infty} b^k + C_0
\]
implies that
\[
 \limsup_{k \to \infty} b^{k} \leq \frac{C_1}{1-\theta},
\]
and thus
\[
 \max_{\alpha} \limsup_{k \to \infty} a^k_\alpha \leq \frac{C_1}{(1-\theta)\ \min_\alpha \{\lambda_\alpha\}} =: C.
\]
If $d = 1$, we have $b^k = a^k_d$, and the iteration
\[
a^{k+1}_d \leq \mu a_d^k + (2-\mu) \Lambda \quad  \mbox{if }\alpha = 1,\ldots,d,
 \]
which allows for the same argument as above.

It remains to pick for $d \geq 2$ some $\lambda_\alpha$ such that \eqref{eq:bkp1wish} holds.

We have from \eqref{eq:akp1iteration}
\begin{align*}
&b^{k+1} = \sum_{\alpha=1}^d \lambda_\alpha\ a^k_\alpha\\
&\leq \sum_{\alpha=1}^{d} \lambda_\alpha\ \brac{\mu a_d^k + (1-\mu) a_{\alpha - 1}^k} + \lambda_0 \Lambda\\
&\leq \mu\ \sum_{\alpha=1}^{d} \lambda_\alpha\ a_d^k + \sum_{\alpha=1}^{d-1} \lambda_{\alpha+1} (1-\mu) a_{\alpha}^k + \lambda_{1} (1-\mu) \Lambda + \sum_{\alpha=1}^d \lambda_\alpha \Lambda\\
&\leq \theta b^k + C\ \Lambda,
\end{align*}
where
\[
 \theta := \max_{\alpha = 1,\ldots,d-1} \left \{\mu\ \brac{1+\sum_{\beta=1}^{d-1} \frac{\lambda_\beta}{\lambda_d}} , (1-\mu) \frac{\lambda_{\alpha+1}}{\lambda_{\alpha}} \right \}.
\]
We need to show that $\theta < 1$: Recall that $d \geq 2$, and thus there is $\tau > 1$ satisfying
\[
     1 < \tau^{d-1} < \frac{\sum_{\beta =0}^{\infty} (1-\mu)^{\beta}}{\sum_{\beta =0}^{d-2} (1-\mu)^{\beta}} = \brac{\mu\ \sum_{\beta =0}^{d-2} (1-\mu)^{\beta}}^{-1}.
\]
Now we can pick our $\lambda_\alpha$:
\[
 \lambda_\alpha = (\tau (1-\mu))^{d-\alpha} \quad \alpha = 1,\ldots,d.
\]
For this choice, certainly
\[
 (1-\mu) \frac{\lambda_{\alpha+1}}{\lambda_{\alpha}} = \frac{1}{\tau} < 1 \quad \mbox{for $\alpha = 1,\ldots,d-1$}.
\]
On the other hand, by the choice of $\tau$,
\begin{align*}
 \frac{\mu}{\mu-1}\sum_{\alpha=1}^{d-1} (\tau (1-\mu))^{d-\alpha} 
 = &\mu\ \sum_{\alpha=1}^{d-1} \tau^{d-\alpha}\ (1-\mu)^{d-\alpha-1}\\
 \leq & \tau^{d-1}\ \mu \sum_{\beta =0}^{d-2} (1-\mu)^{\beta} < 1.
 \end{align*}
This implies that $\theta < 1$ -- which was all that remained to show.
\end{proof}

\subsection{Trees, Sequences, and Iteration}\label{ss:trees}
For the proof of Theorem~\ref{th:uniformconvergence} we use the language of sequences in binary trees: Let $\T$ be the set of all binary trees of finite length. A \emph{tree} is a rooted, connected, undirected, cycle free graph $G = (T,E,0)$, where $T$ is a finite set of vertices, $0 \in T$ (sometimes $t_0$) is the root, $E \subset T \times T$ is a symmetric relation which describes the edges of $G$. 

Since there are no cycles in the graph, the \emph{depth}, or degree, of a vertex $t$, $d(t)$, defined as the number of vertices needed to connect $t$ with $0$ ($d(0) = 0$) is well-defined. With $d(T)$ we denote the maximal depth in the tree $T$.

The \emph{child} of a node $t \in T$ are all vertices $\tilde{t} \in T$ with $d(\tilde{t}) > d(t)$ and $(\tilde{t},t) \in E$. We say that $t$ is a \emph{leaf} of $T$, $t \in {\rm leaf}(T)$, if it has no children. The \emph{parent} of a vertex $t \in T$, denoted with $\operatorname{par}(t)$, is the unique vertex $\tilde{t} \in T$ such that $t$ is a child of $\tilde{t}$.

We are interested in \emph{strictly binary trees}, that is trees whose vertices $t$ have no or exactly two children $\tilde{t}_1$, $\tilde{t}_2$. Moreover, $T$ has to be such that for any vertex $t$ with children $\tilde{t}_1$, $\tilde{t}_2$ there is a left child, which we denote with $t0$, and a right child, which we denote with $t1$, and say $t0 < t1$. Thus any $t \in T$ can be uniquely described by a sequence $t \in \{0,1\}^{d(t)+1}$ with first entry $0$: See Figure~\ref{fig:trees}. The set of these kind of trees shall be called strictly binary trees, and denoted by $\T_2$. All our trees belong to $\T_2$ from now on.

We also introduce a total ordering of a tree, and say that $\tilde{t} < t$ if $d(\tilde{t}) < d(t)$, or if $d(\tilde{t}) = d(t)$ and $\tilde{t}$ is to the left of $t$, see Figure~\ref{fig:treeordering}.

\begin{figure}
\centering
\includegraphics[height=50mm]{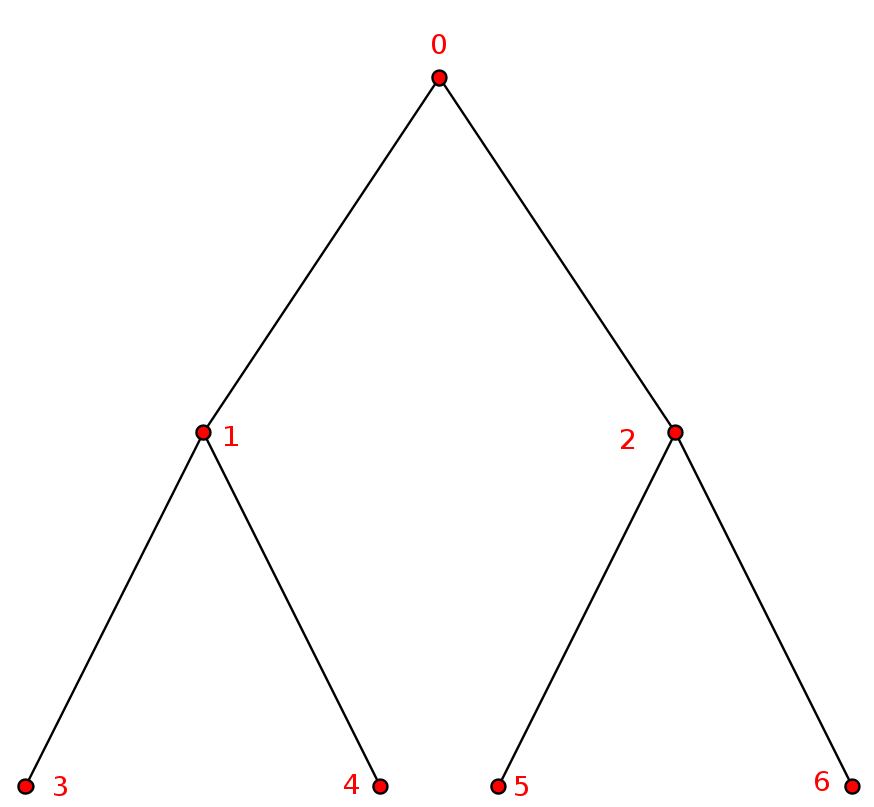}
\caption{The total ordering of trees}
\label{fig:treeordering}
\end{figure}

It will be important for us to know how many left turns $l(t)$ and how many right turns $r(t)$ it needs to reach a vertex $t \in T$ from the root. 
Representing $t$ by a sequence $t \in \{0,1\}^{d(t)+1}$, we have
\[
 l(t) = \#\{0 \in t\} - 1, \quad r(t) = \#\{1 \in t\} = d(t) - l(t).
\]
We will later use $T \in \T_2$ to index sequences $(x_t)_{t \in T} \subset X$, again see Figure~\ref{fig:trees}.

We will need the following estimate, which essentially states that if in a tree $T \in \T_2$ we have a certain relation between all interior nodes and the leafs, then there are relatively few leaves of maximal depth (i.e. in some sense the tree is sparse):
\begin{proposition}\label{pr:depthest}
Given $\mu \in (0,1)$, $C > 0$, for any $\delta > 0$ for $K = \lceil \frac{2C}{\delta} \rceil + 2$ we have the following:

If $T \in \T_2$ satisfies
\begin{equation}\label{eq:trees:sumest}
\sum_{t \not \in {\rm leaf}T} \mu^{l(t)} (1-\mu)^{r(t)}\leq C+ C\ \sum_{t \in {\rm leaf}T} \mu^{l(t)} (1-\mu)^{r(t)},\\
\end{equation}
then
\begin{equation}\label{eq:tree:largeissmall}
 d(T) \geq K\quad \Rightarrow\quad \sum_{d(t) = d(T)} \mu^{l(t)} (1-\mu)^{r(t)} \leq \delta.
\end{equation}
\end{proposition}
\begin{proof}
Fix $\mu$, $C$, $\delta$ and $T$. 

Let us abbreviate
\[
 a_i := \sum_{t: d(t) = i} \mu^{l(t)} (1-\mu)^{r(t)},
\]
and
\[
 b_i := \sum_{t: d(t) = i, t \in {\rm leaf}T} \mu^{l(t)} (1-\mu)^{r(t)},
\]
where $a_i, b_i = 0$ if there is no respective node, and in particular $a_i,b_i = 0$ for $i > d(T)$.

First, let us make the following general remarks about $T \in \T_2$: Since $T$ is strictly binary, observe that nodes $d(t) = i$, for $i \in \N$, come always in pairs. In other words, for any $t \in T$, the fact that $t0 \in T$ is equivalent to the fact that $t1 \in T$. In particular, for any $i \geq 1$:
\begin{align*}
 &\sum_{t: d(t) = i} \mu^{l(t)} (1-\mu)^{r(t)}\\
 &=\sum_{t: d(t) = i, t = \tilde{t}0} \mu^{l(t)} (1-\mu)^{r(t)} + \sum_{t: d(t) = i, t = \tilde{t}1} \mu^{l(t)} (1-\mu)^{r(t)}\\
 &=\sum_{t: d(t) = i, t = \tilde{t}0} \mu^{l(\tilde{t})+1} (1-\mu)^{r(\tilde{t})} + \sum_{t: d(t) = i, t = \tilde{t}1} \mu^{l(\tilde{t})} (1-\mu)^{r(\tilde{t})+1}\\
 &=\sum_{\tilde{t}: d(\tilde{t}) = i-1, \tilde{t}0 \in T} \brac{\mu^{l(\tilde{t})+1} (1-\mu)^{r(\tilde{t})} + \mu^{l(\tilde{t})} (1-\mu)^{r(\tilde{t})+1}}\\
 &=\sum_{\tilde{t}: d(\tilde{t}) = i-1, \tilde{t}0 \in T} \mu^{l(\tilde{t})} (1-\mu)^{r(\tilde{t})}\\
&=\sum_{\tilde{t}: d(\tilde{t}) = i-1} \mu^{l(\tilde{t})} (1-\mu)^{r(\tilde{t})} -  \sum_{\tilde{t} \in {\rm leaf} T, d(\tilde{t}) = i-1} \mu^{l(\tilde{t})} (1-\mu)^{r(\tilde{t})},
\end{align*}
that is
\begin{equation}\label{eq:aip1aibirel}
a_{i+1} =a_{i} - b_{i} \quad \mbox{for $i \geq 0$},
\end{equation}
and in particular,
\[
 a_{i+1} \leq a_{i} \quad \mbox{for $i \geq 0$}.
\]
Thus, the assumption
\begin{equation}\label{eq:assumelarge}
 \sum_{d(t) = d(T)} \mu^{l(t)} (1-\mu)^{r(t)} > \delta,
\end{equation}
implies 
\begin{equation}\label{eq:assumelarge2}
 a_i > \delta \quad \forall i \leq d(T).
\end{equation}
Applying \eqref{eq:aip1aibirel} in \eqref{eq:trees:sumest} we obtain
\[
\begin{array}{lrcl}
&\sum\limits_{i=0}^{d(T)} (a_i - b_i) &\leq& C+ C\ \sum\limits_{i=0}^{d(T)} b_i\\[1em]
\Leftrightarrow\quad &\sum\limits_{i=0}^{d(T)} a_{i+1} &\leq& C+ C\ \sum\limits_{i=0}^{d(T)} (a_i - a_{i+1})\\[1em]
\Leftrightarrow& (C+1)\sum\limits_{i=0}^{d(T)} a_{i+1} &\leq& C+ C\ \sum\limits_{i=0}^{d(T)} a_i\\[1em]
\Leftrightarrow& \sum\limits_{i=1}^{d(T)} a_{i} &\leq& \frac{C}{C+1}+ \frac{C}{C+1}\ \sum\limits_{i=0}^{d(T)} a_i\\[1em]
\Leftrightarrow& \frac{1}{C+1}\sum\limits_{i=1}^{d(T)} a_{i} &\leq& \frac{C}{C+1} + \frac{C}{C+1}\ a_0\\[1em]
\Rightarrow& \sum\limits_{i=1}^{d(T)} a_{i} &\leq& C + C\ a_0 \leq 2C\\[1em]
\end{array}
\]
Plugging \eqref{eq:assumelarge2} in this, we have shown that \eqref{eq:trees:sumest} together with \eqref{eq:assumelarge} leads to
\[
 \delta (d(T)-1) \leq 2C.
\]
This is impossible, if $d(T) \geq K = \lceil \frac{2C}{\delta} \rceil + 2$. That is, if $d(T) \geq K$, the opposite of \eqref{eq:assumelarge} is true, which is the claim \eqref{eq:tree:largeissmall}.
\end{proof}

\section{Boundedness: Proof of Theorem~\ref{th:boundedness}}\label{s:bounded}
In this section we are going to show the following generalized version of Theorem~\ref{th:boundedness}
\begin{theorem}[Boundedness]\label{th:boundednessGen}
Let $(X, Y, \{B(x)\})$ be as in Definition~\ref{def:admissible}.
For any $\Lambda > 0$, $\mu \in (0,1)$, there exists a constant $C = C(\mu,\Lambda,\diam X) > 0$ such that the following holds: 

\begin{itemize}
 \item[(i)] for any
$u_k : X \to \R$, $k \in \N_0$, and such that
\[
 \sup_X u_0 < \infty,
\]
and
\begin{equation}\label{eq:ukp1leq}
 \begin{cases}
  u_{k+1}(x) \leq \mu \sup\limits_{X} u_k + (1-\mu) \inf\limits_{B(x)} u_k + \Lambda \quad &\mbox{if $x \in Y$},\\
 u_{k+1}(x) \leq \Lambda \quad &\mbox{if $x \in X \backslash Y$},
  \end{cases}
\end{equation}
we have
\[
 \limsup_{k \to \infty} \sup_X u_k \leq C.
\]
\item[(ii)] In particular, $u: X \to \R$ satisfying $\sup_X u < \infty$ which is a subsolution, i.e.,
\[
 \begin{cases}
  u(x) \leq \mu \sup\limits_{X} u + (1-\mu) \inf\limits_{B(x)} u + \Lambda \quad &\mbox{if $x \in Y$},\\
 u(x) \leq \Lambda \quad &\mbox{if $x \in X \backslash Y$},
  \end{cases}
\]
actually satisfies
\[
 \sup_X u \leq C,
\]
\item[(iii)] and any $u: X \to \R$ satisfying $\inf_X u > -\infty$ which is a supersolution, i.e.,
\[
 \begin{cases}
  u(x) \geq \mu \sup\limits_{B(x)} u + (1-\mu) \inf\limits_{X} u - \Lambda \quad &\mbox{if $x \in Y$},\\
 u(x) \geq - \Lambda \quad &\mbox{if $x \in X \backslash Y$},
  \end{cases}
\]
actually satisfies
\[
 \inf_X u \geq -C.
\]
\end{itemize}
\end{theorem}
\begin{remark}
It is obvious that for $\mu = 0$ the claims $(i)$ and $(ii)$ still hold. The claim $(iii)$ also holds by switching $\mu$ and $(1-\mu)$.
\end{remark}

\begin{proof}
Assuming (i), the claim of (ii) follows by setting $u_k = u$. The claim (iii) follows from (ii) by replacing $u$ by $-u$, and swapping $\mu$ and $(1-\mu)$.

It remains to show (i). We can assume that w.l.o.g. $u_k \geq 0$; If not, we just replace $u_k$ by $(u_{k})_+ = \max \{u_k,0\}$.

Set $d := \diam X$. We slice our $X$ into subsets $X_\alpha$, for $\alpha \in \{0,1,\ldots,d\}$, which contain all the points which need at most $\alpha$ steps to connect to the boundary via the balls $B(x)$. More precisely, $X_0 := X \backslash Y$, $X_d = X$ and for $\alpha \geq 1$,
\[
 X_\alpha := \left \{ 
 \begin{array}{ll}
 x \in Y\quad :&\mbox{for $\tilde{\alpha} \leq \alpha$ there are $x_0, x_1, \ldots, x_{\tilde{\alpha}} \in X$, such that}\\
	& x_0 = x, \quad x_{\tilde{\alpha}} \in X \backslash Y,\\
	&\mbox{and $x_i \in B(x_{i-1})$ for $1 \leq i \leq \tilde{\alpha}$}
 \end{array}
 \right \}.
\]
Denoting
\[
 a^k_\alpha := \sup_{X_\alpha} u^k \in [0,\infty)
\]
we then obtain the following from \eqref{eq:ukp1leq} 
\[
 a^{k+1}_\alpha \leq \begin{cases}
                 \Lambda \quad & \mbox{if }\alpha = 0,\\
                 \mu a_d^k + (1-\mu) a_{\alpha - 1}^k + \Lambda \quad & \mbox{if }\alpha = 1,\ldots,d.
                 \end{cases}
\]
From the game's point of view, this is essentially assuming that the player who tries to minimize the value function employs the possibly suboptimal strategy of always moving towards the boundary $X \backslash Y$.

From Proposition~\ref{pr:akalphaest}, we have for a constant $C > 0$ depending on $\Lambda$, and diameter $d =\diam X$ and $\mu$, such that
\[
 \limsup_{k \to \infty} \sup_X u^k \leq \max_{\alpha = 0,\ldots,d} \limsup_{k \to \infty} a^{k}_\alpha < C.
\]
\end{proof}

\section{Uniform Convergence and Trees: Proof of Theorem~\ref{th:uniformconvergence}}
The main step in proving Theorem~\ref{th:uniformconvergence} is the following Lemma, which compares $u_k$ to a function $v_k$ which does not depend on $u_0$.
\begin{lemma}\label{la:convergencegen}
Given $\mu \in (0,1)$, $f: Y \to \R$, $F: X \backslash Y \to \R$ with
\[
 \sup_X |F| + \sup_X |f| < \infty,
\]
and
\[
 \inf_Y f > 0.
\]
For any $i \in \N$ there is a function $v_i : X \to R$ such that the following holds:

Assume that $u_k: X \to \R$ for $k \in \N_0$ satisfies
\begin{equation}\label{eq:iterationgen}
 u_{k+1}(x) = \begin{cases}
                f(x) + \mu \sup\limits_{B(x)} u_k + (1-\mu) \inf\limits_{B(x)} u_k \quad &\mbox{if $x \in Y$},\\
		F(x) \quad &\mbox{if $x \in X \backslash Y$},
               \end{cases}
\end{equation}
and
\[
 \sup_X u_0 < \infty.
\]
Then for any $\delta > 0$ there exists $L \in \N$ such that
\[
 \sup_{k,i \in \N_0} \sup_{x \in X} |u_{L+k}(x) - v_{L+i}(x)| \leq \delta.
\]
\end{lemma}

This Lemma implies Theorem~\ref{th:uniformconvergence}: 
\begin{proof}[Proof of Theorem~\ref{th:uniformconvergence}]
First, start the iteration \eqref{eq:iterationgen} with $u_0 := \inf_X F$. Note that we have pointwise monotonicity $u_{k+1}(x) \geq u_k(x)$ for all $k \in \N$, but
\[
\sup_{k \in \N} \sup_X |u_k| < \infty,
\]
by Theorem~\ref{th:boundednessGen}.

So there exists a pointwise limit $u(x) := \sup_k u_k(x): X \to \R$. Note that for all we know so far, $u$ might only be a subsolution. Lemma~\ref{la:convergencegen} tells us, that $\lim_{k \to \infty} u = u$ as a uniform limit: Indeed, fix $\delta > 0$, and let $L$ be from Lemma~\ref{la:convergencegen} so that
\[
 \sup_{k,i \in \N_0} \sup_{x \in X} |u_{L+k}(x) - v_{L+i}(x)| \leq \frac{\delta}{2}.
\]
Fix now $x \in X$, then there exists $K=K(x) \in \N$ such that for any $k \geq 0$,
\[
 |u_{K+k}(x) - u(x)| \leq \frac{\delta}{2}.
\]
In particular, 
\[
 |u(x) - v_{L+i}(x)| \leq |u_{K+L}(x) - u(x)| + |u_{L+K}(x) - v_{L+i}(x)| \leq \delta,
\]
which since it holds for all $x \in X$ implies
\[
 \sup_{i \in \N_0}\sup_X |u(x) - v_{L+i}(x)| \leq \delta.
\]
Especially, for any $\delta > 0$ there is $L \in \N$ such that
\[
 \sup_{k \geq L} \sup_X |u(x) - u_k(x)| \leq  \sup_{k \geq L} \sup_X  |u(x) - v_L(x)| +  \sup_{k \geq L} \sup_X  |v_L(x) - u_k(x)| \leq 2\delta,
\]
and we have uniform convergence.

Let now $\tilde{u}_k$ be any other iteration starting from another $\tilde{u}_0$ with $\sup_X u_0 < \infty$. Again by Lemma~\ref{la:convergencegen}, for the same $v_i$'s there exists another constant $\tilde{L} = \tilde{L}(\delta,\tilde{u}_0)\in \N$ such that
\[
 \sup_X \sup_{k \geq \tilde{L}} |u(x) - \tilde{u}_k(x)| \leq \sup_X |u(x) - v_{\tilde{L}+L}(x)|  + \sup_X \sup_{k \geq \tilde{L}} |\tilde{u}_k(x) - v_{\tilde{L}+L}(x)| \leq 2 \delta.
\]
Thus, starting the iteration from any $\tilde{u}_0$ we have uniform convergence to $u$, and Theorem~\ref{th:uniformconvergence} is proven.
\end{proof}

It remains to prove Lemma~\ref{la:convergencegen}.
\begin{proof}[Proof of Lemma~\ref{la:convergencegen}]
Recall from Section~\ref{ss:trees} the definition of strictly binary trees $T \in \T_2$ , and sequences $(x_t)_{t \in T}$, cf. also Figure~\ref{fig:trees}.
\begin{figure}
\centering
\includegraphics[width=130mm]{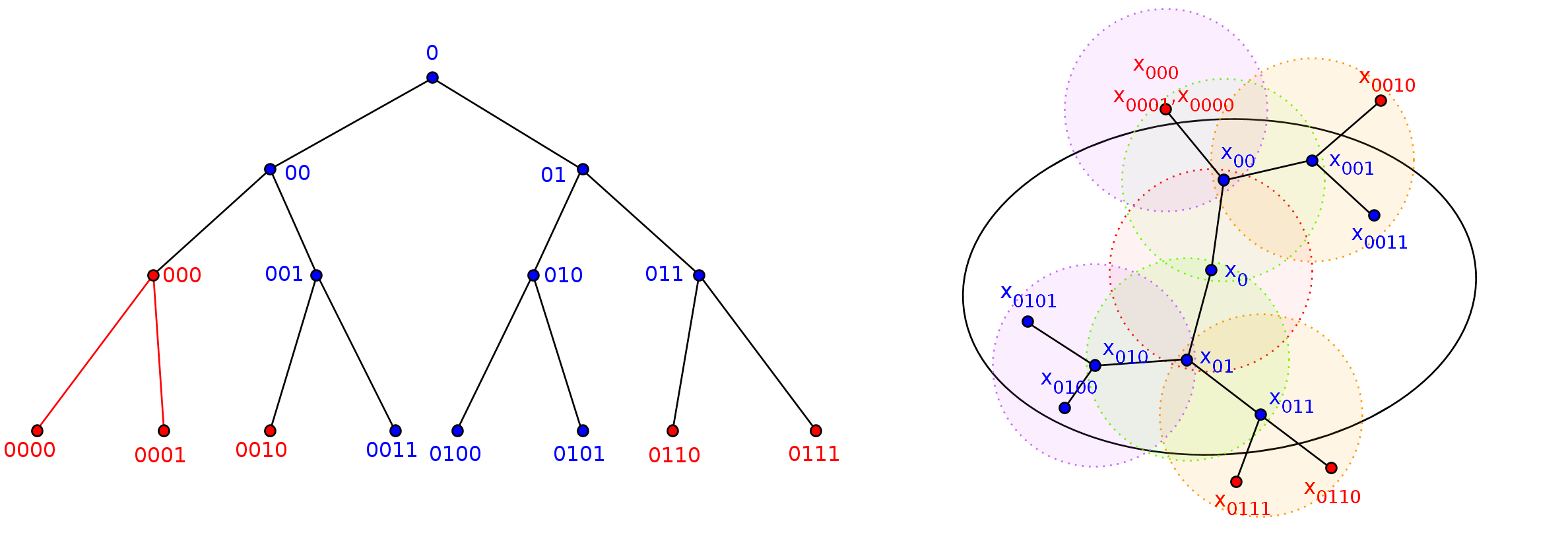}
\caption{A possible run for the depth $L = 3$}
\label{fig:treerun2}
\end{figure}
Using Theorem~\ref{th:boundednessGen}, we can fix $\Lambda > 0$, such that 
\[
  \sup_{X} |F| + \sup_X |f| + \sup_X \sup_k |u_k| \leq \Lambda,
\]
and $\lambda > 0$ such that
\[
 \inf_Y f > \lambda.
\]
For formal reasons it makes sense to set $B(x) = \{x\}$ if $x \in X \backslash Y$ and $f \equiv 0$ on $X \backslash Y$.

Before going into the details, let us describe the general idea. 
It is natural, to compute for some $x$ the subsequent choices which attain $\sup_{B(x)} u$ or $\inf_{B(x)} u$, and try to find estimates on the resulting paths. This is a very natural idea, and has been used for uniqueness arguments, cf. \cite{ArmstrongSmart1} with a probabilistic argument. 
Here, our setup is fully deterministic, and thus we follow both supremum and infimum-paths at the same time, and store this information in trees. The main observation is that these trees have a very special structure. Philosophically, in terms of stochastic games, this is related to the fact that the expectation of the stochastic game terminating is one. 

Let us also remark the following aspect: Our deterministic argument needs to store information in trees, i.e. it needs exponential space in terms of the steps of the game, whereas the probabilistic argument stores its information in structures of polynomial size. Philosophically, one might want to compare this to classical Information Theory, and in particular to the case of non-deterministic and deterministic Turing machines: Every non-deterministic Turing machine can be described by a deterministic Turing machine of exponential size.

\subsection*{Outline of the proof} For fixed $k \in \N$, $L \in \N$, $x = x_0 \in Y$, given $u_k$, we can compute $u_{L+k}$ from the point of view of a game. We introduce Player I, the player that tries to maximize $u_{L+k}$, and Player II, the player that tries to minimize $u_{L+k}$. Starting from $x_0$, Player I chooses his favorite point $x_{00}$, such that \[u_{L-1+k}(x_{00}) \approx \sup_{B(x_0)} u_{L-1+k},\] and Player II picks his point $x_{01} \in B(x_0)$ such that \[u_{L-1+k}(x_{01}) \approx \inf_{B(x_0)} u_{L-1+k}.\] That is to say,
\[
 u_{L+k}(x_0) \approx \mu u_{L-1+k}(x_{00}) + (1-\mu) u_{L-1+k}(x_{01}) + f(x_0).
\]
If $L \geq 1$, we have to go on: If $x_{00}$ lies in the ``boundary`` $X \backslash Y$, the game stops in this branch. If $x_{00}$ is in the ``interior'' $Y$, then again both Players choose their favorite point $x_{000}$ and $x_{001}$ such that
\[
 u_{L-2+k}(x_{000}) \approx \sup_{B(x_{00})} u_{L-2+k}, \quad \mbox{and}\quad u_{L-2+k}(x_{001}) \approx \inf_{B(x_{00})} u_{L-2+k}.
\]
The same we do for $x_{01}$ -- if it is in the boundary $X \backslash Y$, we stop, if it is in the interior $Y$, we pick $x_{010}$ and $x_{011}$. 

Let us look at some examples: in the case where $x_{00}$ and $x_{01}$ are both in the boundary $X \backslash Y$, we have
\[
 u_{L+k}(x_0) \approx \mu F(x_{00}) + (1-\mu) F(x_{01}) + f(x_0),
\]
in the case where $x_{00}$ and $x_{01}$ are both in the interior $Y$, we have
\begin{align*}
 u_{L+k}(x_0) &\approx \mu \brac{\mu u_{L-2+k}(x_{000})+ (1-\mu) u_{L-2+k}(x_{001}) + f(x_{00})}\\
 &\quad + (1-\mu) \brac{\mu u_{L-2+k}(x_{010})+ (1-\mu) u_{L-2+k}(x_{011}) + f(x_{01})}\\
 &\quad + f(x_0)\\
  &\approx \mu^2 u_{L-2+k}(x_{000})\\ 
  & \quad + \mu (1-\mu) u_{L-2+k}(x_{001}) + (1-\mu)\mu u_{L-2+k}(x_{010})\\
  & \quad + (1-\mu)^2 u_{L-2+k}(x_{011}) \\
  & \quad + \mu f(x_{00}) + (1-\mu) f(x_{01}) + f(x_0),
\end{align*}
and in the case where $x_{00}$ is in the interior $Y$ and $x_{01}$ is in the boundary $X \backslash Y$,
\begin{align*}
 u_{L+k}(x_0) &\approx \mu \brac{\mu u_{L-2+k}(x_{000})+ (1-\mu) u_{L-2+k}(x_{001}) + f(x_{00})}\\
 &\quad + (1-\mu) F(x_{01})\\
 &\quad + f(x_0)\\
  &\approx \mu^2 u_{L-2+k}(x_{000})\\ 
  & \quad + \mu (1-\mu) u_{L-2+k}(x_{001}) + (1-\mu) F(x_{01}) \\
  & \quad + \mu f(x_{00}) + f(x_0),
\end{align*}

We iterate this argument $L$ times. We obtain a formula computing $u_{k+L}(x)$ from $u_k$, a tree $T^\ast \in \T_2$ of depth at most $L$, and a sequence indexed by this tree $(x_t)_{t \in T^\ast}$, and we obtain an expression
\[
 u_{L+k}(x) \approx w(x,T^\ast,(x_t)_{t \in T^\ast},u_k).
\]
Our main observation is the following. All these ``optimal'' trees $T^\ast$ have a specific structure: They satisfy the estimate \eqref{eq:titer}, and hence the assumptions of Proposition~\ref{pr:depthest}. On may see this as a kind of comparison principle for game-trees, although we shall not pursue this notion further. Proposition~\ref{pr:depthest} implies that there are actually relatively few leafs of maximal depth $L$ in $T^\ast$. But whenever a game progression $(x_t)_{t \in T^\ast}$ does not end with a leaf of maximal depth, this means that in this branch $x_t$ hits the boundary $X \backslash Y$, where the value of $u_k$ is given by $F$. In other words, in the formula expressing $u_{L+k}(x)$ in terms of $u_k$, most of the terms actually are depending only on the boundary values $F$ and the running costs $f$, and not on $u_k$, i.e., we have
\[
 u_{k+L}(x) \approx w(x,T^\ast,(x_t)_{t \in T^\ast},u_k) =  V(F,f,x) + \mathcal{E}(u_k,x),
\]
where $\mathcal{E}(u_k,x)$ is small. That amounts to saying that $|u_{k+L}(x) - V(F,f,x,L)|$ is small, as desired.

\subsection*{Rigorous Argument}
Given a point $x \in X$, $L \in \N$, we call $B_L(x)$ the \emph{long admissible strategies} of at most $L$ steps: Let $T_L \in \T_2$ be the full tree of depth $L$.
\[
 B_L(x) := \left \{
\begin{array}{lcl}
(x_t)_{t \in T_L}:\ 
&\mbox{(i)}\ &x_0 = x,\\
&\mbox{(ii)}\ &x_t \in B(x_{\operatorname{par}(t)}) \quad \mbox{if $t \neq 0$}
\end{array}
\right \}.
 \]
We call $A_L(x)$ the \emph{short admissible strategies} of at most $L$ steps: 
\[
A_L(x) := \left \{
\begin{array}{lcl}
(T ,(x_t)_{t \in T}):\ 
&\mbox{(i)}\ &T \in \T_2,\ d(T) \leq L,\  x_0 = x,\\
&\mbox{(ii)}\ &\mbox{$t \not \in {\rm leaf}(T)$ $\Rightarrow$ $x_t \not \in X \backslash Y$},\\
&\qquad &\mbox{\ and $x_{t0},x_{t1} \in B(x_t)$},\\
&\mbox{(iii)}\ &\mbox{$t \in {\rm leaf}(T)$ and $d(t) < L$}\\
&\qquad &\mbox{$\Rightarrow$ $x_t \in X\backslash Y$}
\end{array}
\right \}.
\]
Condition (ii) describes that when some $x_t$ has children $x_{t0}$, $x_{t1}$, these have to be in $B(x_t)$, and $x_t$ itself cannot be in the ``boundary'' $X \backslash Y$. The latter means that the game stops if one of the players reaches the boundary.

Condition (iii) tells us that the only way to end a game in less than $L$ steps, is for one of the players to move his point $x_t$ to the boundary $X \backslash Y$.

We write
\[
 A(x) = \bigcup_{L \in \N} A_L(x).
\]
Every long strategy $(x_t)_{t \in T_L} \in B_L(x)$ can be reduced to a unique short strategy $(T^\ast,(x_t)_{t \in T^\ast}) \in A_L(x)$, a process which we depicted in Figure~\ref{fig:treerun2} and Figure~\ref{fig:cuttreerun}: Recall that $B(x) = \{x\}$ whenever $x \in X \backslash Y$. Given $(x_t)_{t \in T_L} \in B_L(x)$, whenever there is $x_t \in X \backslash Y$, then for all successors $\tilde{t}$ of $t$, we have $x_{\tilde{t}} = x_t$. So starting from $T_L$, we erase all successors for all nodes $t$, where $x_t \in X \backslash Y$. The resulting tree, we call $T^\ast$, and the resulting sequence $(x_t)_{t \in T^\ast}$. This reduction is reversible, and any $(T^\ast,(x_t)_{t \in T^\ast}) \in A_L(x)$ can be associated to exactly one $(x_t) \in B_L(x)$.

\begin{figure}
\centering
\includegraphics[width=130mm]{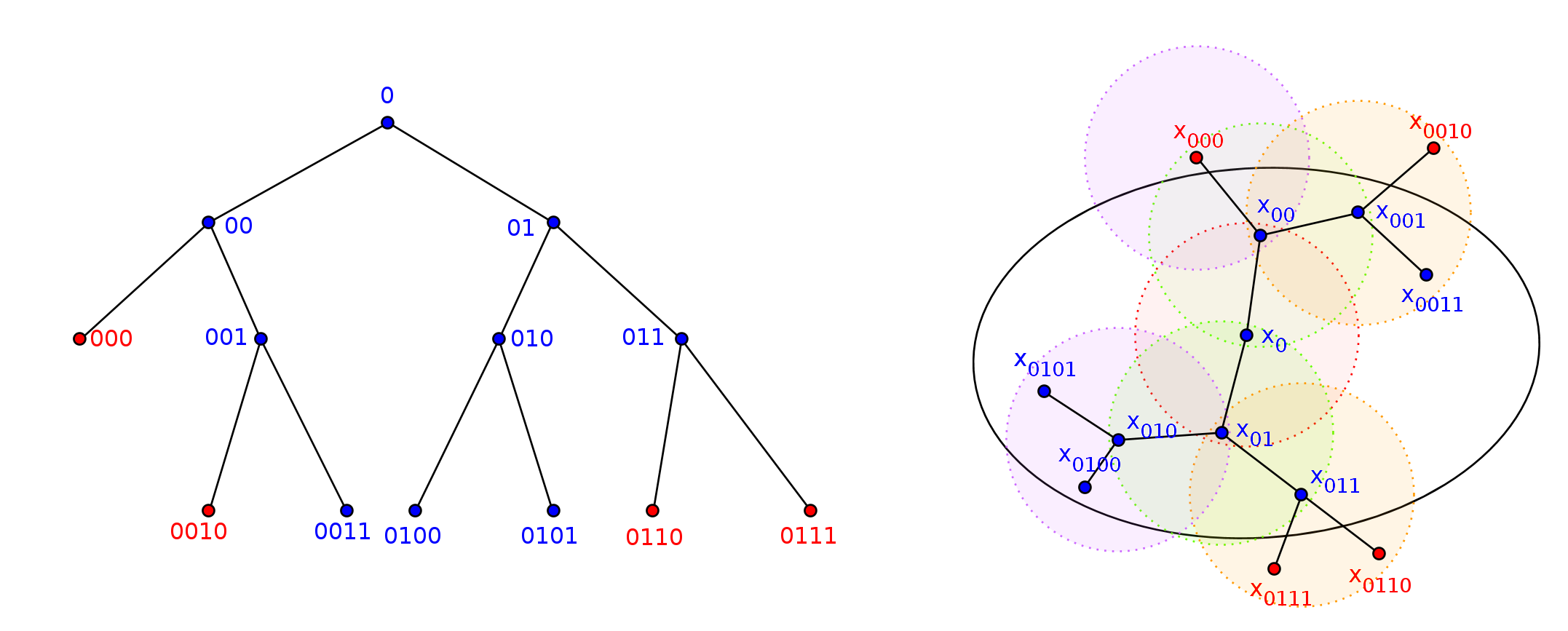}
\caption{The reduction of the tree in Figure~\ref{fig:treerun2} to $T^\ast$, $(x_t)_{t \in T^\ast}$ in $A_3$, where $x_{0000}$, $x_{0001}$ are cut away}
\label{fig:cuttreerun}
\end{figure}

For any $x \in Y$, any tree $T \in \T_2$ and sequence $(x_t)_{t \in T} \in A(x)$ and for any mapping $v : X \to \R$ we set
\begin{align*}
w(x,T,(x_t)_{t \in T},v) := &f(x) + \sum_{t \not \in {\rm leaf}T} \mu^{l(t)} (1-\mu)^{r(t)}\ f(x_t)\\
&+ \sum_{t \in {\rm leaf}T, x_t \in X \backslash Y} \mu^{l(t)} (1-\mu)^{r(t)}\ F(x_t) \\
&+ \sum_{t \in {\rm leaf}T, x_t \in Y} \mu^{l(t)} (1-\mu)^{r(t)}\ v(x_t).
\end{align*}
If $x \in X\backslash Y$ we set $w(x,T,(x_t)_{t \in T},v) := F(x)$.
Also note
\begin{align*}
w(x,T,(x_t)_{t \in T},v) \geq &\lambda +  \lambda\ \sum_{t \not \in {\rm leaf}T} \mu^{l(t)} (1-\mu)^{r(t)}\\
&- \sum_{t \in {\rm leaf}T} \mu^{l(t)} (1-\mu)^{r(t)}\ (\Lambda + \sup_X |v|).
\end{align*}

The operator $w$ should be seen as the $L$-th iteration of \eqref{eq:iterationgen} starting from $v$. We are going to write this as follows, where we recall the construction of $T^\ast$ from above:
\begin{equation}\label{eq:semigroup}
 u_{k+L}(x) = \infsup_{(x_t)_{t \in T_L} \in B_L(x)} w(x,T^\ast,(x_t)_{t \in T^\ast},u_k).
\end{equation}
In order to describe what the right-hand side means, recall the total ordering of the full tree $T_L \in \T_2$, starting from the root $t_0$, and then moving in every layer from left to right. Ordering the tree like this, let us call the $k$th vertex to be $t_k$, $0 \leq k \leq \mathcal{L} := 2^{L+1} - 2$. Then
\[
 \infsup_{(x_t)_{t \in T_L} \in B_L(x)} w(x,T^\ast,(x_t)_{t \in T^\ast},v) =
\]
\[
  \stinfsup^{[t_{\mathcal{L}}]}_{x_{t_{\mathcal{L}}} \in B(x_{\operatorname{par}(t_{\mathcal{L}})})}\ 
 \stinfsup^{[t_{\mathcal{L}-1}]}_{x_{t_{\mathcal{L}-1}} \in B(x_{\operatorname{par}(t_{\mathcal{L}-1})})}\ 
 \ldots 
 \stinfsup^{[t_{2}]}_{x_{t_{2}} \in B(x_{t_{0}})}\ \stinfsup^{[t_{1}]}_{x_{t_1} \in B(x_{t_{0}})} w(x,T^\ast,(x_t)_{t \in T^\ast},v),
\]
where $\stinfsup$ has to be replaced by $\sup$ or $\inf$ according to whether the tree vertex $t_k$ is a left child or a right child of some vertex $\tilde{t}$.
\[
\stinfsup^{[t_{k}]}_{x_{t_{k}} \in B(x_{\operatorname{par}(t_{k})})} = \begin{cases}
                                   \sup\limits_{x_{t_{k}} \in B(x_{\operatorname{par}(t_{k})})} \quad &\mbox{if $t_k = \tilde{t}0$},\\
                                   \inf\limits_{x_{t_{k}} \in B(x_{\operatorname{par}(t_{k})})} \quad &\mbox{if $t_k = \tilde{t}1$}.
                                  \end{cases}
\]
Having defined the right-hand side of \eqref{eq:semigroup}, let us prove it:

It is certainly true for $L = 1$, since by the iteration \eqref{eq:iterationgen}, for $x = x_0 \in Y$,
\begin{align*}
 u_{k+1}(x) &= f(x) + \mu \sup_{x_{00} \in B(x_0)} u_k(x_{00}) + (1-\mu) \inf_{x_{01} B(x)} u_k(x_{01}) \\
	    &= \sup_{x_{00} \in B(x_0)} \inf_{x_{01} \in B(x_0)}  \brac{f(x) + \mu u_k(x_{00}) + (1-\mu) u_k(x_{01})} \\
	    &= \sup_{x_{00} \in B(x_0)} \inf_{x_{01} \in B(x_0)}  w(x,T,(x_t)_{t \in T},u_k).
\end{align*}
for $T = T^\ast = \T_2$, which is the only possible tree in $A_2(x)$, if $x \in Y$.

Now assume \eqref{eq:semigroup} holds for step-sizes of $L-1$, then
\[
 u_{k+L}(x) = \infsup_{(x_t)_{t \in T_{L-1}} \in B_{L-1}(x)} w(x,T^\ast,(x_t)_{t \in T^\ast},u_{k+1}). 
\]
Considering the definition of $w$, we need only to consider $u_{k+1}(x_t)$ for $t$ such that $d(t) = T$, and $x_t \in Y$. For these we use the iteration
\[
 u_{k+1}(x_t) = \sup_{x_{t0} \in B(x_t)} \inf_{x_{t1} \in B(x_t)} \brac{f(x_t) + \mu u_k(x_{t0}) + (1-\mu)u_k(x_{t1})},
\]
and obtain a new, extended tree $\tilde{T}^\ast$ of length at most $L$, and a extended sequence $(x_t)_{t \in \tilde{T}^\ast}$ $\in A_L(x)$, and the resulting formula proves \eqref{eq:semigroup} for $L$.

For any choice of $(T^\ast,(x_t)_{t \in T^\ast}) \in A_L(x)$, such that
\begin{equation}\label{eq:choiceofTclose}
 |w(x,T^\ast,(x_t)_{t \in T^\ast},u_k) - \infsup_{(x_t) \in B_L(x)} w(x,T^\ast,(x_t)_{t \in T^\ast},u_k)| \leq 1,
\end{equation}
we then have that
\begin{align*}
 \Lambda +1 \geq u_{k+L}(x)+1 &\geq \lambda + \sum_{l = 1}^{d(T)} \sum_{t \not \in {\rm leaf}T,\ d(t) = l} \mu^{l(t)} (1-\mu)^{r(t)}\ \lambda\\
&- \sum_{t \in {\rm leaf}T} \mu^{l(t)} (1-\mu)^{r(t)}\ (2\Lambda) ,
\end{align*}
that is for any $(T^\ast,(x_t)_{t \in T^\ast}) \in A_L(x)$ such that \eqref{eq:choiceofTclose} holds, we have
\begin{equation}\label{eq:titer}
\sum_{t \not \in {\rm leaf}T^\ast} \mu^{l(t)} (1-\mu)^{r(t)}\leq C+ C\ \sum_{t \in {\rm leaf}T^\ast} \mu^{l(t)} (1-\mu)^{r(t)},
\end{equation}
for some uniform $C = C(\Lambda,\lambda)$.

That is, if we set the \emph{short, good, admissible strategies} to be $\tilde{A}_{L,C}(x)$,
\[
 \tilde{A}_{L,C}(x) := \{(T^\ast,(x_t)_{t \in T^\ast}) \in A_L(x), \quad \mbox{$T^\ast$ satisfies \eqref{eq:titer}} \},
\]
we have a more precise description of $u_{k+L}$ than that of \eqref{eq:semigroup}. Namely,
\begin{equation}\label{eq:semigroupC}
 u_{k+L}(x) = \infsup_{(x_t)_T,T \in \tilde{A}_{L,C}(x)} w(x,T,(x_t)_{t \in T},u_k).
 \end{equation}
One has to be a little bit careful about the meaning of $\infsup$ in this case: For a sequence $(x_t)_{t \in T}$, and a vertex $t \in T$, we collect the history $(x_{\operatorname{hist},t}) = (x_{\tilde{t}})_{t < \tilde{t}}$ to be all the elements $x_{\tilde{t}}$ for $\tilde{t} < t$. Then
\[
 \infsup_{(x_t)_T,T \in \tilde{A}_{L,C}(x)} w(x,T^\ast,(x_t)_{t \in T^\ast},v) =
\]
\[
  \stinfsup^{[t_{\mathcal{L}}]}_{x_{t_{\mathcal{L}}} \in B(x_{\operatorname{par}(t_{\mathcal{L}})},(x_{\operatorname{hist},t_{\mathcal{L}}}))}\ 
 \ldots 
 \stinfsup^{[t_{1}]}_{x_{t_1} \in B(x_{t_{0}},(x_{\operatorname{hist},t_1}))} w(x,T^\ast,(x_t)_{t \in T^\ast},v),
\]
where the ``balls'' $B( x_{t_{\operatorname{par}(t_k)}} ,(x_{\operatorname{hist},t_k}))$ allow only such elements $y \in B( x_{t_{\operatorname{par}(t_k)}})$, such that picking $y$ as $x_{t_k}$ there still exists at least one sequence $(\tilde{x}_t)_{t \in T_L} \in B_L(x)$ such that $\tilde{x}_t = x_t$ for $t \leq t_k$, and the reduction $T^\ast$ from that sequence $(\tilde{x}_t)_{T}$ satisfies \eqref{eq:titer}. That is to say: Starting from a point $x_t \in Y$, both Players are allowed to take only those points $x_{t0}$ and $x_{t1}$ in $B(x_t)$ such that there is at least one possible way to progress the game with a resulting tree $T^\ast$ that is satisfying \eqref{eq:titer}.

%
%
%

Now we can set,
\[
 v_L(x) = \infsup_{(x_t)_T,T \in \tilde{A}_{L,C}(x)} w(x,T,(x_t)_{t \in T},0),
\]
and observe that by \eqref{eq:semigroupC},
\[
 |u_{k+L}(x) - v_L(x)| \leq \sup_{(x_t),T \in \tilde{A}_{L,C}(x)} |w(x,T,(x_t)_{t \in T},u_k)-w(x,T,(x_t)_{t \in T},0)|.
\]
Moreover,
\begin{align*}
 &|w(x,T,(x_t)_{t \in T},u_k)-w(x,T,(x_t)_{t \in T},0)|\\
 &\leq \sum_{t \in {\rm leaf}T, x_t \in Y} \mu^{l(t)} (1-\mu)^{r(t)}\ |u_k(x_t)| \\
 &\leq \Lambda \sum_{t \in {\rm leaf}T, x_t \in Y} \mu^{l(t)} (1-\mu)^{r(t)}\\
  &= \Lambda \sum_{t \in {\rm leaf}T, d(t)=d(T)=L} \mu^{l(t)} (1-\mu)^{r(t)}.
\end{align*}
By Proposition~\ref{pr:depthest} there exists $L_0 > 0$ (depending only on the constants involved), such that the latter is smaller than $\frac{\delta}{2}$ for any kind of tree of length greater or equal to $L_0$ satisfying \eqref{eq:titer}.

Thus we have shown
\begin{equation}\label{eq:equation1ukpL}
 \sup_{k} \sup_{x \in X} |u_{k+L}(x) -v_{L}(x)| \leq \frac{\delta}{2} \quad \mbox{for any $L \geq L_0$}.
\end{equation}
On the other hand, note that $v_i$ is the solution to \eqref{eq:iterationgen} starting from $v_0 = 0$. In particular, $\sup_i \sup_X |v_i| < \infty$ by Theorem~\ref{th:boundednessGen}. Repeating the argument from above, we obtain the existence of $K_0 > 0$ such that
\begin{equation}\label{eq:equation2vKpi}
\sup_{i} \sup_{x \in X} |v_{K+i}(x) -v_K(x)| \leq \frac{\delta}{2} \quad \mbox{for all $K \geq K_0$}.
\end{equation}
Combining \eqref{eq:equation1ukpL} and \eqref{eq:equation2vKpi}, we arrive at
\[
 \sup_{k,i} \sup_X |u_{k+L+K} -v_{L+K+i}| \leq \delta.
\]
\end{proof}

\bibliographystyle{plain}%
\bibliography{bib}%

\end{document}